\documentclass{amsart}

 \usepackage{amsxtra}
 \usepackage{amsthm}
 \usepackage{amsmath}
 \usepackage{amsthm}
 \usepackage{amsfonts}
 \usepackage{amssymb}
 \usepackage[all,cmtip]{xy}

\allowdisplaybreaks[4]

\input xy
\xyoption{all}
\newtheoremstyle{mytheorem}{.5em}{.5em}%
     {\it}%         Body font
     {}%         Indent amount (empty = no indent, \parindent = para indent)
     {}% Thm head font
     {}%        Punctuation after thm head
     {.5em}%     Space after thm head (\newline = linebreak)
     {#2 \thmname{ \bf{#1}.} \thmnote{\it{#3}.}}%         Thm head spec

\theoremstyle{mytheorem}
\newtheorem{theorem}{Theorem}[section]
\newtheorem{lemma}[theorem]{Lemma}
\newtheorem{corollary}[theorem]{Corollary}
\newtheorem{prop}[theorem]{Proposition}

\newtheoremstyle{note}{.5em}{.5em}%
     {}%         Body font
     {}%         Indent amount (empty = no indent, \parindent = para indent)
     {}% Thm head font
     {}%        Punctuation after thm head
     {.5em}%     Space after thm head (\newline = linebreak)
     {#2 \thmname{ \bf{#1}.} \thmnote{\it{#3}.}}%         Thm head spec

\theoremstyle{note}

\newtheorem{example}[theorem]{Example}
\newtheorem{remark}[theorem]{Remark}
\newtheorem{definition}[theorem]{Definition}

\newcommand{\lbr}{[\hspace{-1.5pt}[}  % left bracket  
\newcommand{\rbr}{]\hspace{-1.5pt}]}  % right bracket  
\newcommand{\calJ}{\mathcal J}

\newcommand{\La}{\Lambda}           % Lambda
\newcommand{\la}{\lambda}            % lambda
\newcommand{\IF}{\mathcal{I}_F}    % the augmentation ideal for F
 
\newcommand{\ttt}{\mathtt t}     % torsion index
\newcommand{\ep}{\epsilon}
\newcommand{\Th}{\Theta}
\newcommand{\De}{\Delta}
\newcommand{\Dev}{{\mathbf{\mathit{\Delta}}}}
\newcommand{\ka}{\kappa}
\newcommand{\de}{\delta}
\newcommand{\al}{\alpha}

\newcommand{\ga}{\gamma}
\newcommand{\Ga}{\Gamma}

\newcommand{\Z}{\mathbb Z}
\newcommand{\Q}{\mathbb Q}
\newcommand{\tilF}{{\tilde{F}}}

\newcommand{\HF}{\mathbf{H}_F}
\newcommand{\WHF}{{\hat{\mathbf{H}}_F}}  % the stable hecke algebra
\newcommand{\CHF}{{\mathcal{H}_F}} 
\newcommand{\DF}{{\mathbf{D}_F}}
\newcommand{\FA}{{R\lbr \La\rbr _F}}
\newcommand{\FFA}{{R_F\lbr \La\rbr _F}}
\newcommand{\FFAK}{{R_F\lbr \La\rbr _F^\kappa}}
\newcommand{\CR}{{\mathcal{R}_F}}

\newcommand{\GR}{{\mathcal{G}r}}
\newcommand{\CQ}{{\mathcal{Q}^{'}}}
\newcommand{\CQW}{{\mathcal{Q}_W^{'}}}

\newcommand{\LTA}{\mathcal{T}}
\newcommand{\End}{\operatorname{End}}

\title{On the formal affine Hecke algebra}

\date{\today}

\author{Changlong Zhong}

\address{Changlong Zhong, Fields Institute and 
University of Ottawa}
\email{zhongusc@gmail.com}

\begin{document}
\maketitle

\begin{abstract} We study the action of the formal affine Hecke algebra on the formal group algebra, and show that the the formal affine Hecke algebra has a basis indexed by the Weyl group as a module over the formal group algebra. We also define a concept called the normal formal group law, which we use to simplify the relations of the generators of the formal affine Demazure algebra and the formal affine Hecke algebra.
\end{abstract}

\section{Introduction}

Let $R$ be  a commutative ring  and $F$ be a formal group law over $R$. For  the  lattice $\La$ of a root datum, the formal group algebra $R\lbr \La\rbr _F$ was defined in \cite{CPZ} and \cite{CZZ}. It can be thought  as an algebraic replacement of the algebraic oriented cohomology theory (in the sense of Levine--Morel \cite{LM}) of the classifying space $BT$ of the torus $T$. For each root $\al$, they define the so-called  Demazure--Lusztig type operator $\De^F_{\al}$ acting on the formal group algebra $R\lbr \La\rbr _F$. The subring of endomorphisms of $R\lbr \La\rbr _F$ generated by these operators were used to study the algebraic oriented cohomology of flag varieties. Furthermore, generalizing  the construction of the nil-Hecke algebra of \cite{KK86} and \cite{KK90},  in \cite{HMSZ}, these operators were constructed as elements (called the formal Demazure elements $X_{\al}^F$) in the semidirect product of $R\lbr \La\rbr _F$ and the group ring $R[W]$ of the Weyl group $W$. The algebra generated by the formal Demazure elements $\Dev_\al$ is called the formal affine Demazure algebra $\DF$. Such construction provides a more convenient way to study the equivariant cohomology theory. For instance, in \cite{CZZ} and \cite{CZZ2}, it is shown that the $R\lbr \La\rbr _F$-dual of $\DF$ can be view as the $T$-equivariant algebraic oriented cohomology of complete flag variety. Such algebraic language has been proved to be useful, for example, see \cite{MZZ} and \cite{Zho}. 

On the other hand, in the classical theory of Hecke algebra  the Demzaure element $\Dev_{\al}$  can be thought as degeneration of generator $T_{s_{\al}}$ of affine Hecke algebra (see \cite{Lus}, \cite{CG} or \cite{Gin}). Following this direction, in \cite{HMSZ}, for arbitrary formal group law $F$, the formal affine Hecke algebra $\HF$ was defined to be generated by $T^F_{s_{\al_i}}$ for all simple roots $\al_i$. After this generalization, the classical affine Hecke algebra becomes special case of $\HF$ when $F$ is the multiplicative formal group law. The reader could consult \cite{HMSZ} for a more detailed introduction of classic background. 

For the formal affine Hecke algebra $\HF$, the quadratic relations and the braid type relations between the generators were  constructed, but the latter ones come with some mysterious auxiliary coefficients. As a result, the coefficient ring of $\HF$ is certain enlarged  version of the formal group algebra. The goal of this paper is to study these coefficients and to show that this enlargement is not necessary. More precisely, we show in Theorem \ref{thm:mainhecke} that  for each $w\in W$, we fix a reduced sequence $I_w$ of $w$, then  the formal affine Hecke algebra is free with basis $\{T_{I_w}\}_{w\in W}$ as a module over the formal group algebra (with the $\kappa_\al$'s inverted, see the theorem for precise statement). We also show in Theorem \ref{thm:action} that it could be identified as the endomorphism subring of the formal group algebra generated by the formal group algebra and the Demazure--Lusztig type operators. Moreover, in Corollary \ref{cor:center}, we  provide a description of the center of $\HF$. These results provide a solid foundation for the formal affine Hecke algebra.

The second part of this paper considers formal affine Hecke algebra of some specific types of formal group laws. For a (one-dimensional commutative) formal group law 
$$F(x,y)=x+y+\sum_{i,j\ge 1} a_{ij}x^iy^j, \quad a_{ij}\in R$$
such that $a:=a_{11}$ is invertible in $R$,  we define a new formal group law $\tilF$ and call it the associated normalization of $F$. This new formal group law   is isomorphic  to $F$ but with more explicit and simple formal inverse element.  With the normalization, we simplify the notations and relations of the formal (affine) Demazure algebra $\DF$ and the formal (affine) Hecke algebra $\HF$. Moreover, we show that (Theorem \ref{thm:hecke}), when $a$ is invertible,  the coefficients appearing in the braid type relations of the generators $T^F_{s_{\al_i}}$ belong to the formal group algebra for root systems  of all types (except for type $G_2$).  This provides a different proof of Theorem \ref{thm:mainhecke} with more explicit coefficients. Moreover, this method does not use the torsion index, so it could be generalized to the setting when one considers the corresponding formal affine Hecke algebra of Kac-Moody algebras. This will be treated in a forthcoming paper. Moreover, we hope that such simplification will make the study of the formal affine Hecke algebra easier. 

The paper is organized as follows: In Section 2 we recall the notation of formal group algebra and the Demazure--Lusztig type operators, in Section 3 we recall the definition of formal affine Demazure algebra and the formal affine Hecke algebra. In Section 4 we consider the action of the formal affine Hecke algebra on the formal group algebra, and prove the main result (Theorem \ref{thm:mainhecke}). In Section 5 we define the notion of  normal formal group laws, and use it to simplify the notations of the formal affine Demazure algebra and the formal affine Hecke algebra in Section 6.

\section*{Acknowledgements}
The author would like to thank A. Savage and K. Zainoulline for helpful conversations, and Baptiste Calm\`es for suggestions on simplifying the second part of the paper. He would like to thank the referee for many suggestions on polishing this paper. The author was supported by the Fields Institute and NSERC grant of K. Zainoulline.

%%%%%%%%%%%%%%%%%%%%%%%
\section{Formal group algebra and the Demazure--Lusztig type operators}
Let $R$ be a commutative ring with identity and let $R\lbr x \rbr$ be the ring of formal power series. Let $F$ be a one-dimensional formal group law over $R$, i.e., a formal power series $F(x,y)=x+y+\sum_{i,j\ge 1}a_{ij}x^iy^j\in R\lbr x, y\rbr $ satisfying the following properties 
$$F(x,F(y,z))=F(F(x,y),z), \quad F(x,y)=F(y,x), \quad  F(x,0)=x.$$
We sometimes write $x+_Fy$ for $F(x,y)$, and denote the formal inverse of $x$ by $\imath_Fx$, which  has leading terms $-x+ax^2$ with $a:=a_{11}\in R$. Denote
$$\mu_F(x)=\tfrac{\imath_Fx}{-x}\in R\lbr x\rbr  , \quad\kappa^F(x)=\tfrac{1}{x}+\tfrac{1}{\imath_Fx}\in R\lbr x\rbr .$$
\begin{lemma} \cite[Lemma 4.3] {HMSZ} For any FGL $F$ over $R$, $\kappa^F(x)=0$ if and only if $\mu_F(x)=1$, and if and only if $F(x,y)=(x+y)k(x,y)$ for some $k(x, y)\in R\lbr x,y\rbr $.
\end{lemma}
If $\kappa^F(x)\neq 0$, since $\imath_Fx=-x+ax^2+x^3g(x)$ with $g(x)\in R\lbr x \rbr$,   
$$\kappa^F(x)=\tfrac{ax^2+x^3g(x)}{x(-x+ax^2+x^3g(x))}=\tfrac{a+xg(x)}{-1+ax+x^2g(x)},$$
we know that  $\kappa^F(x)$ is invertible in $R\lbr x\rbr$  if and only if $a$ is invertible  in $R$, and $\kappa^F(x)$ is regular if $a$ is regular \cite[Lemma 12.3]{CZZ}.

\begin{definition}\label{def:kappa} We say $\kappa^F=0$ (resp. $\ka^F$ is regular) if $\ka^F(x)=0$  (resp. if $\kappa^F(x)$ is regular).
\end{definition}

\begin{example} \cite[Example 2.2]{HMSZ}
(1) The additive FGL $F_a=x+y$  and the Lorentz FGL $F_l=\tfrac{x+y}{1+\beta xy}$ with $\beta\in R$ satisfy that $\kappa^F=0$.

(2) The multiplicative FGL $F_m(x,y)=x+y+a xy$ with $a$ regular in $R$ satisfy that $\kappa^F$ is regular, and  we have $\imath_{F_m}(x)=\tfrac{u}{-a u-1}$, $\mu_{F_m}(x)=\tfrac{1}{1+a u}$ and $\kappa^{F_m}(x)=-a$.

(3) Let $E$ be the elliptic curve defined by the Tate model:
$$v=u^3+a_1uv+a_2u^2v+a_3v^2+a_4uv^2+a_6v^3,$$
then the elliptic FGL defined by $E$ is 
$$F_e(u,v)=u+v-a_1uv-a_2(u^2v+uv^2)-2a_3(u^3v+uv^3)+(a_1a_2-3a_3)u^2v^2+...$$
over the ring $R=\Z[a_1,a_2,a_3,a_4,a_6]$, and $\kappa^{F_e}$ is regular.

(4) There is a universal formal group law $F_u$ over the Lazard ring $\mathbb{L}$, which is a polynomial ring over $\Z$ with infinitely number of variables. Moreover, any formal group law $(R,F)$ is given by a unique ring homomorphism $\phi:\mathbb{L}\to R$ such that $F(x,y)=\phi (F_u(x,y))$. 
\end{example}

\begin{definition} Let $\La$ be a free abelian group of rank $n$.  Consider the polynomial ring $R[x_{\La}]$ in variables $x_\la$ with $\la\in \La$. Let $$\ep:R[x_{\La}]\to R, \quad x_{\la}\mapsto 0$$ be the augmentation map, and let $R\lbr x_{\La}\rbr $ be the ($\ker\ep$)-adic completion of $R[x_{\La}]$. Let $\calJ_F$ be the closure of the ideal of $R\lbr x_\La\rbr $ generated by $x_0$ and elements of the form $x_{\la_1+\la_2}-F(x_{\la_1}, x_{\la_2})$ for all $\la_1, \la_2\in \La$. Here $x_0\in R[x_{\La}]$ is the element determined by the zero element of $\La$. The \textit{formal group algebra } $R\lbr \La\rbr _F$ is defined to be the quotient
$$R\lbr \La\rbr _F=R\lbr x_{\La}\rbr /\calJ_F.$$

The augmentation map  induces a ring homomorphism $ \ep: R\lbr \La\rbr _F\to R$, and we denote the kernel by $\IF$. Then we have a filtration of  ideals
$$R\lbr \La\rbr _F=\IF^0\supseteq \IF^1\supseteq \IF^2\supseteq \cdots$$
and the associated graded ring
$$Gr_R(\La, F):=\bigoplus_{i=0}^{\infty}\IF^i/\IF^{i+1}.$$
For simplicity, we assume that $\IF^{-n}=R\lbr \La\rbr _F$ for $n>0$.
By \cite[Lemma 4.2]{CPZ}, $Gr_R(\La, F)$ is isomorphic to $S_R^*(\La)$,  the symmetric algebra. The isomorphism sends $\prod x_{\la_i}$ to $\prod \la_i$. Moreover, if $\{w_i\}_{i=1}^n$ is a basis of $\La$,  then $\FA\cong R\lbr w_1,...,w_n\rbr $.

\end{definition}

\medskip
We recall the notation of root datum in \cite[\S2]{CZZ}. Let $\La$ be a lattice and $\Sigma\subsetneq \La$ be a non-empty finite subset. A \textit{root datum} is an embedding $\Sigma\hookrightarrow \La^\vee, \al\mapsto \al^\vee$ of $\Sigma$ into  the dual of $\La$ satisfying certain axioms. The elements of $\Sigma$ are called \textit{roots}, and the sublattice $\La_r$ of $\La$ generated by $\Sigma$ is called the \textit{root lattice}. The dimension of $\La_\Q$ is called the \textit{rank} of the root datum. The root datum is called \textit{irreducible} if it is not direct sum of root data of smaller ranks. The set $\La_w:=\{w\in \La_ \Q|\al^\vee(w)\in \Z\text{ for all } \al\in \Sigma\}$ is called the \textit{weight lattice}. We assume that the root datum is  \textit{semisimple}, i.e., $\La_r$ and $\La$ have the same rank. Then we have $\La_r\subseteq \La\subseteq \La_w$. If $\La=\La_w$ (resp. $\La=\La_r$), then the root datum is called \textit{simply connected} (resp. \textit{adjoint}), and is denoted by $\mathcal{D}_n^{sc}$ (resp. $\mathcal{D}_n^{ad}$), where $\mathcal{D}=A,B,C,D,E,F,G$ is one of the Dynkin types and $n$ is the rank. An irreducible root datum is uniquely determined by its Dynkin type and the lattice $\La$.

We assume that the rank of the root datum is $n$, and fix a basis $\Pi:=\{\al_1,...,\al_n\}\subsetneq \Sigma$. Denote $[n]=\{1,...,n\}$. We call $\Pi$ the set of \textit{simple roots}. Then each root $\al$ is an integral linear combination of $\al_i$ with coefficients either all  positive or all  negative, and we have a decomposition $\Sigma=\Sigma_-\sqcup \Sigma_+$ where $\Sigma_+$ (resp. $\Sigma_-$) is called the set of \textit{positive} (resp. \textit{negative}) roots. The reflection $\La \to \La, \la\mapsto \la-\al^\vee( \la)\al$ is denoted by $s_\al$, and the  group $W$ generated by $s_\al, \al\in \Sigma$ is called the \textit{Weyl group}.  Let $\ell$ be the length function with respect to the set of simple roots $\Pi$,  and let $m_{ij}$ be the order of $s_is_j$ in $W$ for $i\neq j$. Let $w_0$ be the longest element in $W$, and let $N$ be its length.

Let $\ttt$ be the \textit{torsion index} of the root datum, whose prime divisors are the torsion primes of the associated simply connected root datum together with the prime divisors of $|\La_w/\La|$.  Clearly $|\La_w/\La|$ divides $|\La_w/\La_r|$. We copy the table from \cite[\S2]{CZZ} (see \cite{Dem} for the definition of torsion primes). 
\begin{center}
\begin{tabular}{c|c|c|c|c|c|c|c|c|c}
Type (s. connected) & $A_l$ & $B_l, l\ge 3$ & $C_l$ & $D_l, l\ge 4$ & $G_2$ & $F_4$ & $E_6$ & $E_7$ & $E_8$\\
$|\La_w/\La_r|$& $l+1$ & 2 & 2 & 4 & 1 & 1 & 3 & 2 & 1\\
Torsion primes & $\emptyset$ & 2 & $\emptyset$ & 2 & 2 & 2,3 & 2,3 & 2,3 & 2,3,5
\end{tabular}
\end{center} 

\begin{definition}\cite[Definition. 4.4]{CZZ}\label{def:sigreg} Let $\Sigma\hookrightarrow \La^\vee$ be a root datum. We say that $R\lbr \La\rbr _F$ is $\Sigma$-regular if for each $\al\in \Sigma$, the element $x_\al$ is regular in $R\lbr \La\rbr _F$.
\end{definition}

By \cite[Lemma 2.2]{CZZ}, $R\lbr \La\rbr _F$ is $\Sigma$-regular if $2$ is regular in $R$ or if the root datum does not contain an irreducible component $C_n^{sc}$. In this paper we always assume that $R\lbr \La\rbr _F$ is $\Sigma$-regular. 

\medskip

The action of the Weyl group $W$  on $\La$ extends to an action on $\FA$ by 
$s_\al(x_\la)=x_{s_\al(\la)}.$
For each root $\al$, define the following operator 
$$\Delta^F_\al(u)=\tfrac{u-s_\al(u)}{x_\al}.$$
By \cite[Corollary 3.4]{CPZ}, for any $u\in \FA$, one has $\Delta^F_\al(u)\in \FA$. In other words, it is an $R$-linear operator on $\FA$, called the \textit{Demazure--Lusztig type operator}. For simplicity, we will write $s_i=s_{\al_i}$ and $\De_i=\De^F_{\al_i}$. If $I=(i_1,...,i_m)$ is a sequence in $[n]$, we define $l(I)=m$ and 
$$\De_I=\De_{i_1}\circ \cdots \circ \De_{i_m}.$$
 We say that $I$ is reduced if $w(I):=s_{i_1}...s_{i_m}$ is reduced in $W$, i.e., when $\ell(w(I))=l(I)$. For $w\in W$, we can choose a reduced decomposition $w=s_{i_1}\ldots s_{i_m}$, then $I_w=(i_1,...,i_m)$ is called a reduced sequence of $w$.  The definition of $\De_{I_w}$ will depend on the choice of $I_w$, unless $F$ is the additive or a multiplicative FGL. We denote $\kappa_\al^F=\kappa^F(x_\al)=\tfrac{1}{x_\al}+\tfrac{1}{x_{-\al}}$.

\begin{lemma}\cite{CPZ} \label{lem:DLa}Let $u,v\in R\lbr \La\rbr _F$, $\al\in \Sigma$ and $I$ be a sequence in $[n]$, then 
\begin{itemize}
\item[(1)] $\De_\al(1)=0$ and $x_\al\De_{\al}(u)=u-s_\al(u)$.
\item[(2)] $\De_\al^2(u)=\kappa_\al\De_\al(u)$, $s_\al\De_\al(u)=-\De_{-\al}(u)$ and $\De_\al s_\al(u)=-\De_\al(u)$.
\item[(3)] $\De_\al(uv)=\De_\al(u)v+s_\al(u)\De_\al(v)=\De_\al(u)v+\De_\al(v)u-x_\al\De_\al(u)\De_\al(v)$.
\item[(4)] For any $w\in W$, $w\De_\al w^{-1}=\De_{w(\al)}$.
\item[(5)] $\De_I$ is $R\lbr \La\rbr _F^W$-linear and functorial with respect to maps of rings $f:R\to R'$.
\item[(6)] We have $\De_I(\IF^i)\subseteq \IF^{i-l(I)}$. Moreover, $\De_I(\IF^i)\subseteq\IF^{i-l(I)+1}$ if $I$ is not reduced.
\end{itemize}
\end{lemma}

The operator $\De_\al$ induces an operator $\GR\De_\al$ on $\GR^*_R(\La,F)$ of degree $-1$. If we denote by $\De_\al^{F_a}$ the classical Demazure--Lusztig operator on $S_R^*(\La)$, i.e., 
$$\De_\al^{F_a}(\la)=\tfrac{\la-s_\al(\la)}{\al},\quad \la\in \La,$$ then by \cite[Proposition 4.11]{CPZ}, $\GR\De_\al$ exchanges with the operator $\De_\al^{F_a}$ via the isomorphism $\GR^*_R(\La, F)\cong S_R^*(\La)$. 

Let $w_0$ be the longest element in $W$ and let $I_0$ be a reduced sequence of $w_0$. Let $N$ be its length.  By \cite[\S5]{CPZ}, there exists $u_0\in \IF^N$ such that $\De_{I_0}(u_0)\equiv \ttt\mod \IF$. Moreover, 

\begin{lemma} \label{lem:u0} The element $u_0$ satisfies the following properties:
\begin{itemize}
\item[(1)] If $l(I)\le N$, then $\ep\De_{I}(u_0)=0$ unless $l(I)=N$ and $I$ is reduced, in which case $\De_I(u_0)\equiv \ttt\mod\IF$.
\item[(2)] If $\ttt$ is invertible in $R$, then the matrix $(\De_{I_v}\De_{I_w}(u_0))_{(v,w)\in W\times W}$ is invertible. 
\end{itemize}
\end{lemma}

%%%%%%%%%%%%%%%%%%%%%%%%%%%%%%%%%%
\section{Formal affine Demazure algebra and formal affine Hecke algebra}
In this section we recall the definitions and properties of the formal affine Demazure algebra $\DF$ and the formal affine Hecke algebra $\HF$ in \cite{HMSZ} and \cite{CZZ}.

Let $Q^F=\FA[\tfrac{1}{x_\al}|\al\in \Sigma]$, i.e., the localization of $\FA$ at $\{x_\al|\al\in \Sigma\}$. The action of $W$ on $\FA$ extends to an action on $Q^F$. Define the \textit{twisted formal group algebra} to be the smash product $Q_W^F:=Q^F\sharp_R R[W]$. That is,  $Q_W^F$ is $Q^F\otimes_R R[W]$ as an $R$-module, and the multiplication is given by 
$$q\de_w \cdot  q'\de_{w'}=qw(q')\de_{ww'}, \quad w,w'\in W ,  q,q'\in Q^F.$$
Notice that $Q^F_W$ is a free (left and right) $Q^F$-module with basis $\{\de_w\}_{w\in W}$, but it is not a $Q^F$-algebra since  the embedding $Q^F\to Q^F_W, q\mapsto q\de_e$ is not central in $Q^F_W$. Here $e\in W$ is the identity in $W$, and we denote $\de_e$ by 1.

For each root $\al$,  we  define the \textit{Demazure element} 
$$X_{\al}^F:= \tfrac{1}{x_{\al}}(1-\de_{s_{\al}})=\tfrac{1}{x_{\al}}-\de_{s_{\al}}\tfrac{1}{x_{-{\al}}}\in Q_W^F.$$

\begin{definition}Define the \textit{formal affine Demazure algebra} $\DF$ to be  the $R$-subalgebra of $Q_W^F$ generated by $\FA$ and $X_{\al}^F$ for all root $\al$.
\end{definition}

If there is no confusion, we will denote  $X_i=X_{\al_i}^F$ and $\de_i=\de_{s_i}$. 
If $I=(i_1,...,i_l)$ is a sequence in  $[n]$, we write $$X_I=X_{i_1i_2...i_l}=X_{i_1}\cdots X_{i_l}.$$ If $I_w$ is a reduced sequence of $w$, then $X_{I_w}$ depends on the choice of $I_w$, unless $F$ is an additive  or  a multiplicative FGL \cite[Theorem 3.10]{CPZ}.
For any $i\neq j\in [n]$, we denote $x_{\pm i}=x_{\pm\al_i}$, $x_{i\pm j}=x_{\al_i\pm \al_j}$ and define
\begin{eqnarray}
\kappa_{ij} = \kappa^F_{ij} =\tfrac{1}{x_{i+j}x_j}-\tfrac{1}{x_{i+j}x_{-i}}-\tfrac{1}{x_ix_j},
\end{eqnarray} 
which belongs to $\FA$ by  \cite[Lemma 6.7]{HMSZ}.

In the following proposition, part (7) for any root datum that does not contain an irreducible component of type  $G_2$ and part (1)--(6) for all root data are proved in \cite[\S6]{HMSZ}. Part (7) and part (8) for arbitrary root datum are proved in \cite[Proposition 7.7 and Lemma 5.8]{CZZ}. The latter proof is  uniform for root datum of any type, but does not give the relations explicitly. 

\begin{prop}\label{prop:demazure} Suppose that $\FA$ is $\Sigma$-regular. The algebra $\DF$ satisfies the following properties:

(1) For any $q\in Q^F$, we have 
$qX_\al=X_\al s_\al(q)+\De_\al(q)$.

(2) For any root $\al$, we have $X_\al^2=\kappa_\al X_\al$.

(3) For any $i\neq j$, if $(\al_i^\vee,\al_j)=0$, so that $m_{ij}=2$, then $X_{ij}=X_{ji}$.

(4) For any $i\neq j$, if $(\al_i^\vee,\al_j)=(\al_j^\vee, \al_i)=-1$, so that $m_{ij}=3$, then $$X_{jij}-X_{iji}=\kappa_{ij}X_i-\kappa_{ji}X_{j}.$$

(5) For any $i\neq j$, if $(\al_i^\vee, \al_j)=-1$ and $(\al_j^\vee, \al_i)=-2$, so that $m_{ij}=4$, then 
\begin{gather*}X_{jiji}-X_{ijij}\\
=(\kappa_{i+j, j}+\kappa_{ij})X_{ij}-(\kappa_{i+2j, -j}+\kappa_{ji})X_{ji}+\De_i(\kappa_{i+j,j}+\kappa_{ij})\cdot X_j+\De_{j}(\kappa_{i+2j, -j}+\kappa_{ji})X_i.
\end{gather*}

(6) If the root datum does not contain any component of type $G_2$, then as an $R$-algebra,  $\DF$ is generated by $\FA$ and $X_i, i\in [n]$ with relations  (1) -- (5)  above.

(7) For each $w\in W$, choose a reduced decomposition of $w$ and define $X_{I_w}$ correspondingly. Then   $\{X_{I_w}\}_{w\in W}$ is a basis of $\DF$ as a left $\FA$-module. 

(8) For any root datum, $\DF$ is the $R$-subalgebra of $Q_W$ generated by $\FA$ and by $X_\al$ for all roots $\al$.
\end{prop}
\begin{remark} The paper \cite{HMSZ} treats $\DF$ as a right $\FA$-module while we think  it as a left $\FA$-module. One could use the identity in Proposition \ref{prop:demazure}.(1) to transform the coefficients from left to right. 
\end{remark}
\medskip
\begin{definition}\cite{HMSZ} The \textit{Hecke algebra} $H$ of $W$ is the $\Z[t,t^{-1}]$-algebra generated by $T_i$ with $i\in [n]$ with the relations:
\begin{itemize}
\item[(1)] $(T_i+t^{-1})(T_i-t)=0$ for all $i\in [n]$.
\item[(2)] The braid relation $T_iT_jT_i\cdots =T_jT_iT_j\cdots$ ($m_{ij}$ factors on both sides) for any $i\neq j$ with $s_is_j$ of order $m_{ij}$ in $W$.
\end{itemize} 
The \textit{affine Hecke algebra} is $\Z[t,t^{-1}][\La_w]\otimes_{\Z[t,t^{-1}]}H$ containing both $\Z[t,t^{-1}][\La_w]$ and $H$ as subalgebras, and the commuting relations are 
$$e^\al T_i-T_i e^{s_i(\la)}=(t-t^{-1})\tfrac{e^\la-e^{s_i(\la)}}{1-e^{-\al_i}}.$$
\end{definition}
\medskip
In the remaining part of this paper, we assume that either $\kappa^F=0$ or $\ka^F$  is regular (the latter being satisfied if  $a$ is regular in $R$). In the latter case we have $\ka_\al^F$ is regular for any root $\al$, because $\FA$ is assumed to be $\Sigma$-regular. We work under this assumption because if $\kappa^F$ is not regular, then the $\FFAK$ defined below is 0 and everything in the remaining part of this paper becomes trivial. 

We recall the definition of the formal affine Hecke algebra. Fix a free abelian group $\Ga$ of rank 1 and denote $R_F=R\lbr \Ga\rbr _F$. Let $\gamma$ be a generator of $\Gamma$ and denote by $x_\ga$ the corresponding element in $R_F$. Define  $\Th_F=\mu_F(x_\ga)-\mu_F(x_{-\ga})$. For simplicity, we write $\mu=\mu_F({x_\ga})$ and $\Th=\Th_F$. We fix some notation:
$$ R_F\lbr \La\rbr _F^\kappa=\left\{\begin{array}{ll}R_F\lbr \La\rbr _F & \text{if }\kappa^F=0,\\
                                            R_F\lbr \La\rbr _F[\tfrac{1}{\kappa_\al}|\al\in \Sigma] & \text{if }\kappa^F \text{ is regular}. \end{array}\right.~ \vartheta_\al:=\left\{\begin{array}{ll}      
                               {2x_\ga}, & \text{if }\kappa^F=0,\\
                               \tfrac{\Th}{\kappa_\al}, & \text{if }\kappa^F \text{ is regular}.
                               \end{array}\right. ~$$
  Denote $\vartheta_i=\vartheta_{\al_i}$. Let 
  $(Q^F)'= R_F\lbr \La\rbr _F^\kappa[\tfrac{1}{x_\al }|\al\in \Sigma]$.  Recall that  $F$ satisfies that either $\kappa^F=0$ or $\kappa^F$ is regular, so  $\FFAK\subseteq (Q^F)'$. Let $(Q_W^F)'$ be the corresponding twisted formal group algebra defined over $R_F$, i.e., $(Q^F_W)'=(Q^F)'\sharp_{R_F} R_F[W]$. For each  root $\al$, define the element 
 $$T_\al^F:=\vartheta_{\al} X_\al+\mu\de_\al\in (Q_W^F)'.$$
 \begin{definition} Define the \textit{formal affine Hecke algebra} $\HF$ to be  the $R_F$-subalgebra of $(Q_W^F)'$ generated by $\FFAK$ and $T_{\al_i}^F$ for all simple roots $\al_i$.
 \end{definition}
 \begin{remark} The original definition of $\DF$ and $\HF$ in \cite{HMSZ} were for torsion-free ring $R$ and for $\La=\La_w$ , i.e., for simply connected root data. The paper \cite{CZZ} generalizes the definition of $\DF$ to any  any ring $R$ and any root datum such that  $\FA$ is $\Sigma$-regular (Definition \ref{def:sigreg}), and we mimic this generalization in the present paper for $\HF$. 
 \end{remark}

If the FGL $F$ is clear in the context, we will write $Q'=(Q^F)'$,  $Q_W'=(Q_W^F)'$ and $T_i=T_{\al_i}^F$ for simplicity. Notice that $\Th$, $\mu$ and $\kappa_\al$ are invariant under $s_\al$, so they commute with $X_\al$ and $\de_\al$. We also have  $\de_wT_{\al_i} \de_{w^{-1}}=T_{w(\al_i)}$. 

For a sequence $I=(i_1,...,i_l)$ in $[n]$, we denote 
$$T_I=T_{i_1...i_l}=T_{i_1}\cdot ...\cdot T_{i_l}.$$
For each $w$, we choose a reduced sequence $I_w$. Then  $T_{I_w}$ depends on the choice of $I_w$ unless $F$ is a multiplicative or the additive FGL.  If $\ka^F$ is regular, let 
$$\kappa'_{ij} = (\ka^F_{ij})' = \tfrac{1}{\kappa_{i+j}x_{i+j}\kappa_jx_j}-\tfrac{1}{\ka_{i+j}x_{i+j}\ka_{i}x_{-i}}-\tfrac{1}{\ka_ix_i\ka_jx_j}\in Q'.$$
Unlike $\kappa^F_{ij}$, a priori it is unclear whether $(\kappa^F_{ij})'$ belongs to $\FFAK$ (see Corollary \ref{cor:kappaij'} or Lemma \ref{lemma:2}).
 
 The following proposition is from \cite[\S8]{HMSZ}, which is based on the assumption that $R$ is torsion free and the root datum is simply connected. It is not difficult to check that its proof is purely formal and depends only on the assumption that $\FFA$ is $\Sigma$-regular.  
\begin{prop}\label{prop:hecke}Suppose that  $\FFA$ is $\Sigma$-regular. For each $w\in W$, fix a reduced decomposition and define $T_{I_w}$ correspondingly. The algebra $\HF$  satisfies the following properties:

\begin{itemize}
\item[(1)] For any $q\in Q'$ and $i\in [n]$, we have $qT_i-T_i s_i(q)=\vartheta_i\De_i(q)$.
\item[(2)]   $T_i^2=\Th T_i+1$. In particular, if $\kappa^F=0$, then $T_i^2=1$.
\item[(3)] For any $i\neq j$, we have $$\underbrace{T_{j}T_iT_j\cdots }_{m_{ij} \text{ terms}}-\underbrace{T_iT_jT_i\cdots}_{m_{ij} \text{ terms}} =\sum_{\ell(w)\le m_{ij}-2}\tau^{I_w}_{ij}T_{I_w}$$
for some $\tau^{I_w}_{ij}\in Q'$. 

\item[(4)] For any $i\neq j$, if $(\al_i^\vee, \al_j)=0$, so that $m_{ij}=2$, then $T_{ij}=T_{ji}$.
\item[(5)] for any $i\neq j$, if $(\al_i^\vee, \al_j)=(\al_j^\vee, \al_i)=-1$, so that $m_{ij}=3$, then 
$$T_{jij}-T_{iji}=\ka'_{ij}\Th^2(T_i-T_j).$$
Moreover, $\ka'_{ij}=\ka_{ji}'$.
\end{itemize}
Define $R_F\lbr \La\rbr _F^\sim$ to be the subring of $Q'$ generated by $R_F\lbr \La\rbr _F^\kappa$, and the elements $v(\tau^{I_w}_{ij})$ and $\De_k(\tau^{I_w}_{ij})$ for any $i,j,k\in I$ and $v,w\in W, \ell(w)\le m_{ij}-2$. Let $\HF^\sim $ be the $R_F$-algebra generated by $R_F\lbr \La\rbr _F^\sim$ and $T_i$ for all $i$. 
\begin{itemize}
\item [(6)] Suppose that  the root datum is simply laced, i.e., $m_{ij}\le 3$ for all $i\neq j\in [n]$.  As an $R_F$-algebra, $\HF^\sim$ is generated by $R_F\lbr \La\rbr _F^\sim$ and  $T_i$ for all $i\in [n]$ with relations (1) -- (5) above.
\item[(7)] The set $\{T_{I_w}\}_{w\in W}$ forms a basis of $\HF^\sim$ as a left $R_F\lbr \La\rbr _F^\sim$-module. 

\end{itemize}
\end{prop}
The goal of this paper is to study the structure of $\HF$ and the mysterious coefficients $\tau_{ij}^{I_w}$. In Corollary \ref{cor:kappaij'}, we show that they actually belong to $\FFAK$, hence $R_F\lbr \La\rbr _F^\sim=\FFAK$ and $\HF=\HF^\sim$.  
\begin{remark}Similar as the case of $\DF$, we treat $\HF$ as a left module. The proof of Proposition \ref{prop:hecke}.(3) is similar to that of \cite[Proposition 6.8]{HMSZ}, but in general  the relationship between the coefficients  $\tau^{I_w}_{ij}$ and the coefficients in \textit{loc.it.} is unclear to us. 
\end{remark}
\begin{remark}We have $T_\al=(\mu-\tfrac{\vartheta_\al}{x_\al})\de_\al+\tfrac{\vartheta}{x_\al}$, and $\mu-\tfrac{\vartheta_\al}{x_\al}$ is not invertible in $(Q^F)'$. In other words, $\de_\al\not\in \HF$ and the set $\{T_{I_w}\}_{w\in W}$ is not a basis of $(Q_W^F)'$. This is different from the case if one considers the set $\{X_{I_w}\}_{w\in W}$ (cf.  \cite[Corollary 5.6]{CZZ}). Besides, $\de_\al\not\in \HF$ also implies that $T_\al\not\in \HF$ unless $\al$ is a simple root. Therefore, the definition of $\HF$ depends on the choice of simple roots. Note that $\DF$ does not depend on such choice \cite[Lemma 5.8]{CZZ}.
\end{remark}

%\begin{corollary} \label{cor:commute}For any sequence $I$ and any $u\in \FFAK$, we have 
%$$T_Iu=\sum_{E\subseteq I}\phi_{I,E}(u)T_{E}$$
%for some $\phi_{I,E}(u)\in \FFAK$. Here $E\subseteq I$ means $E$ is a subsequence of $I$.  Moreover, if %$I=(i_1,...,i_l)$, then $\phi_{I,I}(u)=s_{i_1}\cdots s_{i_l}(u)$.
%\end{corollary}
%\begin{proof} If $I=(i)$, this is Proposition \ref{prop:hecke}.(1). For $I$ of length $\ge 2$, it follows from the same idea as the proof in \cite[Lemma 9.3]{CZZ}.
%\end{proof}
For any sequence $I$, let $W(I)$ consist of $w\in W$ such that there exists a reduced sequence $I_w$ which is also a subsequence of $I$. If $I=I_u$ is a reduced sequence of $u$, then $W(I_u)=\{w\le u\}$.
\begin{corollary}\label{cor:commute2} We fix a set of reduced sequences $\{I_w\}_{w\in W}$.  

\begin{itemize}
\item[(a)] For any sequence $I$, we have $T_I=\sum_{w\in W(I)}q_wT_{I_w}$ with $q_w\in Q'$.  

\item[(b)] For any $x\in Q'$ we have
$$T_{I_w}x=\sum_{u\le w}\phi_{I_w,I_u}(x)T_{I_u}$$
such that $\phi_{I_w,I_u}(x)\in Q'$ and $\phi_{I_w,I_w}(x)=w(x)$. 

\end{itemize}
\end{corollary}
\begin{proof}
\noindent (a) This follows from Proposition \ref{prop:hecke} and the embedding $R_F\lbr\La\rbr_F^\sim\subseteq Q'. $

\noindent (b) From Proposition \ref{prop:hecke}.(1)-(3) we know that $T_{I_w}x=\sum_{E\subset I_w}\psi_{I_w,E}(x)T_E$ where $\psi_{I_w,E}(x)\in Q'$. From part (a) we know that $T_E=\sum_{u\in W(E)}q_{E,u}T_{I_u}$, so 
\[
T_{I_w}x=\sum_{E\subset I_w}\sum_{u\in W(E)}\psi_{I_w,E}(x)q_{E,u}T_{I_u}=\sum_{u\le w}\phi_{I_w,I_u}(x)T_{I_u}. 
\]
Moreover, it follows from Proposition \ref{prop:hecke}.(1)-(3) that $\phi_{I_w,I_w}(x)=w(x)$.
\end{proof}
%%%%%%%%%%%%%%%%%%%%%%%%%%%%%%%%%55
\section{Structure theorem of formal affine Hecke algebra}
In the present section we introduce the action of the formal affine Hecke algebra $\HF$ on the formal group algebra, and  prove the main structure theorems of $\HF$. The main idea is to generalize the idea of \cite[Proposition 6.6]{CPZ} and \cite[Proposition 6.2]{CZZ} to the setting of $\HF$ by enlarging the formal group algebra.

Unless otherwise stated, in the section we assume that $2\ttt$ is regular in $R$, and either  $\ka^F=0$ or $a$ is regular in $R$, where $a=a_{11}$ is the coefficient of $xy$ in the power series $F(x,y)$. In the latter case, the leading term of $\Th$ is $2ax_{\ga}$, so $\Th$ is regular in $R\lbr \Ga\rbr _F$ (see \cite[Lemma 12.3]{CZZ}). We define 
$$\CR=\left\{\begin{array}{ll}R\lbr \Ga\rbr _F[\tfrac{1}{a\ttt \Th }] & \text{if }a \text{ is regular},\\
                R\lbr \Ga\rbr _F[\tfrac{1}{2\ttt x_\ga}] & \text{if }\kappa^F=0.\end{array}\right. $$          
$$\CQ:=\CR\lbr \La\rbr _F[\tfrac{1}{x_\al}|\al\in \Sigma], \quad \CQW=\CQ\sharp_{\CR}\CR[W],$$
$$\mathcal{T}_\al:=\vartheta_\al X_\al+\mu \de_\al \in \CQW,$$
where the product structure of $\CQW$ is similar as  that of $Q_W^F$. 

   \begin{definition} We define $\CHF$ to be the $\CR$-subalgebra of $\CQW$ generated by $\CR\lbr \La\rbr _F$ and $\LTA_i$ for all simple roots $\al_i$.
   \end{definition}

\begin{lemma} \label{lem:invertelem} If $a$ is regular, then  $\kappa_\al$ is invertible in $\CR\lbr \La\rbr _F$ for any root $\al$, and in particular $\vartheta_\al\in \CR\lbr\La\rbr_F$. For a general $F$,  $\mu-\tfrac{\vartheta_\al}{x_\al}\in \CQ$ is invertible and its inverse is in  $\CR\lbr \La\rbr _F$.
 \end{lemma}
 \begin{proof} Since $\kappa_\al$ has leading term $a$, it is invertible in $\CR\lbr \La\rbr _F$. 
Concerning $\mu-\tfrac{\vartheta_\al}{x_\al}$, we only need to consider $\vartheta_\al-\mu x_\al$. Since $\CR\lbr\La\rbr_F$ is a power series ring and subtracting a multiple of $x_\al$ from $\vartheta_\al$ does not change its constant term,  it suffices to show that $\vartheta_\al$ is invertible in $\CR\lbr\La\rbr_F$. If $\kappa^F=0$, then $\vartheta_\al=2x_\ga,$ and if $a$ is regular, then $\vartheta_\al=\tfrac{\Th}{\kappa_\al}$, so in both cases  $\vartheta_\al$ is invertible in $\CR\lbr\La\rbr_F$.
\end{proof}            

\begin{corollary} \label{cor:HFCHF} We have inclusions $\FFAK\subseteq \CR\lbr\La\rbr_F$ and $\HF\subseteq \CHF$.
\end{corollary}
\begin{proof}The first inclusion follows from Lemma \ref{lem:invertelem}, and consequently we have inclusions $(Q_W)'\subseteq \CQW$, which induces the second inclusion by mapping $T_i$ to $\mathcal{T}_i$. 
\end{proof}
 \begin{corollary} We have an isomorphism of $\CR\lbr \La\rbr _F$-modules:
   $\CHF\cong  \CR\lbr \La\rbr _F\otimes_{R\lbr \La\rbr _F}\DF$, and  $\{X_{I_{w}}\}_{w\in W}$ is a basis of $\CHF$.
   \end{corollary}
\begin{proof} Note that $\de_\al=1-x_\al X_\al$, so 
$$\mathcal{T}_\al=\vartheta_\al X_\al+\mu \de_\al=\mu+(\vartheta_\al-\mu x_\al)X_\al.$$
By Lemma \ref{lem:invertelem}, $\vartheta_\al-\mu x_\al$ is invertible in $\CR\lbr\La\rbr_F$, so $\CHF$ is the $\CR$-subalgebra of $\CQW$ generated by $X_\al$ for all  roots $\al$ and $\CR\lbr \La\rbr _F$, hence $\CHF\cong \CR\lbr \La\rbr _F\otimes_{R\lbr \La\rbr _F}\DF$. In particular, $\{X_{I_w}\}_{w\in W}$ is a basis of $\CHF$ as a left $\CR\lbr \La\rbr _F$-module.
\end{proof}

\begin{lemma} \label{lem:transition} For each $v\in W$, we have
$$T_{I_v}=\sum_{w\in W}a_{v,w}\de_w\in Q_W'$$
for some $a_{v,w}\in Q'$ such that 
\begin{itemize}
\item[(1)] $a_{v,w}=0$ unless $w\le v$ with respect to the Bruhat order on $W$,
\item[(2)]  $a_{v,v}=\prod_{\al\in v(\Sigma_-)\cap \Sigma_{+}}(\mu-\tfrac{\vartheta_\al}{x_\al})$.
\end{itemize}
Similar result holds for $\mathcal{T}_{I_v}$. Moreover, $\{\mathcal{T}_{I_{w}}\}_{w\in W}$ is a basis of $\CQW$ as a left $\CQ$-module. 
\end{lemma}
\begin{proof} Note that  $T_i=\tfrac{\vartheta_i}{x_i}+(\mu-\tfrac{\vartheta_i}{x_i})\de_i$. Then by direct computation similar as the one in \cite[Lemma 5.4]{CZZ}, we obtain the conclusion. The property for $\mathcal{T}_{I_v}$ follows similarly.

For the last part, notice that  the transition matrix from the basis $\{\de_w\}_{w\in W}$ to the set $\{\mathcal{T}_{I_w}\}_{w\in W}$ is lower triangular with $\mu-\tfrac{\vartheta_\al}{x_\al}$ on the diagonal, so by Lemma \ref{lem:invertelem},  this matrix is invertible. Hence,  the conclusion follows.
\end{proof}

\medskip

For each root $\al$, we define an operator in $\End_{R_F}(R_F\lbr \La\rbr _F^\kappa)$:
$$\tau_\al^F(u)=\vartheta_\al\De_\al(u)+\mu s_\al(u)\in R_F\lbr \La\rbr _F^\kappa,\quad u\in R_F\lbr \La\rbr _F^\kappa. $$
Similarly we can define  $\tau_\al^F$ in $\End_{\CR}(\CR\lbr \La\rbr _F)$. For each sequence $I=(i_1,...,i_m)$ in $[n]$, define 
$$\tau^F_{I}=\tau^F_{\al_{i_1}}\circ\cdots \circ \tau^F_{\al_{i_m}}.$$ Again, if $I_w$ is a reduced sequence of $w\in W$, the operator $\tau^F_{I_w}$ will depend on $I_w$ unless $F$ is the additive or a multiplicative FGL.

\begin{lemma}The following formulas hold for any $u,v\in \CR\lbr \La\rbr _F$ and $w\in W$:
\begin{itemize}
\item[(1)] $\tau_\al(1)=\mu$, $\tau_\al(uv)=\tfrac{1}{\mu x_\al-\vartheta_\al}[x_\al\tau_\al(u)\tau_\al(v)-\vartheta_\al(\tau_\al(u)v+\tau_\al(v)u)+\vartheta_\al\mu u v].$
\item[(2)]  $\tau_\al^2(u)=\Th\tau_\al(u)+u$ and in particular,  $\tau_\al^2(u)=u$ if $\kappa^F=0$. 
\item[(3)] $\tau_\al (s_\al(u))=-\tau_\al(u)+\mu(u+s_\al(u)).$
\item[(4)] $w\tau_\al w^{-1}(u)=\tau_{w(\al)}(u)$.
\item[(5)] The operator $\tau_\al$ is $\CR\lbr \La\rbr _F^W$-linear.
\end{itemize}
\end{lemma}

\begin{proof} We obtain (1) and (3) by direct computation. For (2), it follows from computation similar as in \cite[Lemma 8.8]{HMSZ}. The parts (4) and (5) follow from corresponding properties of $\De_\al$ in Lemma \ref{lem:DLa}.
\end{proof}
Let $\IF$ be the kernel of the augmentation map $\ep: \CR\lbr \La\rbr _F\to \CR$.  We have 
$$\ep(\vartheta_\al)=
\left\{\begin{array}{ll}\tfrac{\Th}{a} & \text{if }a \text{ is regular},\\
2x_\ga & \text{if }\kappa^F=0.
\end{array}\right.$$ 
Denote $\vartheta=\ep(\vartheta_\al)$. By $\De_\al(\IF^i)\subseteq \IF^{i-1}$ and $s_\al(\IF^i)=\IF^i$, so $\tau_\al(\IF^i)\subseteq \IF^{i-1}$. The operator $\tau_\al$ induces an operator on the associated graded ring $$\GR\tau_\al: \GR^*_{\CR}(\La, F)\to \GR^*_{\CR}(\La, F)$$ of degree $-1$. Let  $\GR\De_\al:\GR^*_{\CR}(\La,F)\to \GR^*_{\CR}(\La,F)$ be the graded version of $\De_\al$ (for example, see \S2).

\begin{lemma} \label{lem:GRtau}For any sequence $I$, we have:

(1) $\GR \tau_I=\vartheta^{l(I)}\GR\De_I $.

(2) $\tau_I(\IF^i)\subseteq \IF^{i-l(I)}$. If $I$ is not reduced, then $\tau_I(\IF^i)\subseteq \IF^{i-l(I)+1}$.

\end{lemma}
\begin{proof} (1) If $I=(i_1)$, we have
$$\GR\tau_{i_1}=\ep(\vartheta_{i_1})\GR\De_{i_1}=\vartheta\GR\De_{i_1}.$$
Now let $I=(i_1,...,i_m)$ and let $I'=(i_2,...,i_m)$. Recursively, we have 
$$\GR\tau_I(u)=\GR\tau_{{i_1}}\circ \GR\tau_{I'}(u)=\vartheta\GR\De_{i_1}(\vartheta^{m-1}\GR\De_{I'}(u))=\vartheta^m\GR\De_I(u).$$ 

(2) By Lemma \ref{lem:DLa}.(6), $\GR\De_I$ has degree $-l(I)$, so by (1), $\tau_I(\IF^i)\subseteq \IF^{i-l(I)}$. If $I$ is not reduced, then by Lemma \ref{lem:DLa}.(6), $\GR\De_I=0$, hence $\GR\tau_I=0$, so $\tau_I(\IF^i)\subseteq \IF^{i-l(I)+1}$.
\end{proof}

Recall that by Lemma \ref{lem:u0}, there exists $u_0\in \IF^N$ such that $\De_{I_0}(u_0)\equiv\ttt\mod\IF$. The following lemma uses the proof of \cite[Proposition 6.6]{CPZ}. 
  
\begin{lemma}\label{lem:u0tau}
The element $u_0$ satisfies the following property: If $l(I)\le N$, then 
$$\ep \tau_I(u_0)=\left\{\begin{array}{ll} \vartheta^{N}\ttt & \text{if } I \text{ is reduced and }l(I)=N,\\
                               0 & \text{otherwise.}\end{array}\right.$$
Consequently, the matrix $(\tau_{I_v}\tau_{I_w}(u_0))_{(v,w)\in W\times W}$  is invertible in $\CR\lbr \La\rbr _F$.                               
\end{lemma}      
\begin{proof} If $l(I)<N$, or if $l(I)=N$ and  $I$ not reduced, we have $\tau_I(u_0)\in\IF$, so $\ep\tau_I(u_0)=0$. If $l(I)=N$ and $I$ reduced, then $\GR\tau_I(u_0)=\vartheta^N\GR\De_{I}(u_0)=\vartheta^N\ttt$, so $\tau_I(u_0)\equiv\vartheta^N\ttt\mod\IF.$

Let us prove the second part.  Since $\IF=\ker \ep$ is contained in the Jacobson radical of $\CR\lbr \La\rbr _F$,  it suffices to show that the determinant of $(\ep\tau_{I_v}\tau_{I_w}(u_0))_{(v,w)\in W\times W}$ is invertible in $\CR$. If we order the $v$'s by increasing length and the $w$'s by decreasing length, then by (1), the matrix $(\ep\tau_{I_v}\tau_{I_w}(u_0))_{(v,w)\in W\times W}$ is lower triangular with $\vartheta^N\ttt$ on the diagonal, so its determinant is  a power of $\vartheta^N \ttt$, so the matrix is invertible in $\CR$.
\end{proof}

There is an action of $Q_W'$ on $Q'$ by 
$$(q\de_w)\cdot q'=qw(q'), \quad q,q'\in Q', ~w\in W, $$
and clearly $T_\al\cdot q=\tau_\al(q)$.  Similarly, $\CQW$ acts on $\CQ$ and $\mathcal{T}_\al\cdot q=\tau_\al(q)$.  They induce  a map of $R_F$-algebras
$$\HF\to \End_{R_F}(R_F\lbr \La\rbr _F^\kappa), \quad T_\al\mapsto \tau_\al,$$
and   a map of $\CR$-algebras $\CHF\to \End_{\CR}(\CR\lbr\La\rbr_F), \mathcal{T}_\al\mapsto \tau_\al$. If $f\in \HF\subseteq \CHF$, then denote its image in $\CHF$  by $f'$, we then have a commutative diagram
\begin{equation}\label{eq:HFCHF}\xymatrix{\FFAK\ar[r]^f\ar@{}[d]|-*[@]{\subseteq}& \FFAK\ar@{}[d]|-*[@]{\subseteq}\\
\CR\lbr\La\rbr_F \ar[r]^{f'} & \CR\lbr\La\rbr_F.}
\end{equation}

\begin{definition} We define the left $\FFAK$-submodule of $Q_W'$ by 
$$\WHF=\{f=\sum_{w\in W}q_wT_{I_w}|q_w\in Q' \text{ and } f\cdot R_F\lbr \La\rbr _F^\kappa\subseteq R_F\lbr \La\rbr _F^\kappa\}.$$ 
By Corollary \ref{cor:commute2}, $\WHF$ is an $R_F$-algebra, and  contains $T_i$ and $\FFAK$, so $\WHF\supseteq \HF$. 
\end{definition}

We are ready to prove the main result of this paper:

\begin{theorem} \label{thm:mainhecke} Suppose that $2\ttt$ is regular in $R$, and either $2\in R^\times$ or the root datum does not have any irreducible component $C_n^{sc}, n\ge 1$. Suppose that either $\kappa^F=0$ or  $a=a_{11}$ is regular in $R$. Then   $\WHF=\HF$ with basis $\{T_{I_w}\}_{w\in W}$   as a left $\FFAK$-module.
\end{theorem}
\begin{proof} Let $f=\sum_{w\in W}q_w{T}_{I_w}\in \WHF$ with $q_w\in Q'$. If we could show that $q_w\in \FFAK$ for all $w\in W$, then $\{T_{I_w}\}_{w\in W}$ is a basis of $\WHF$ as left $\FFAK$-module, so $\WHF\subseteq \HF$. 

Denote the image of $f$ along the inclusion $(Q_W)'\subseteq \CQW$ by $f'$. Applying $f'$ to $\tau_{I_v}(u_0)\in \CR\lbr \La\rbr _F$ for all $v\in W$, we get a system of linear equations $$\sum_{w\in W}q_w\tau_{I_w}\tau_{I_v}(u_0)=f'\cdot \tau_{I_v}(u_0)\in \CR\lbr \La\rbr _F.$$ By Lemma \ref{lem:u0tau}, we know that the matrix $(\tau_{I_w}\tau_{I_v}(u_0))_{(w,v)\in W\times W}$ is invertible, so $q_w\in \CR\lbr \La\rbr _F$. 

If $\kappa^F=0$, then we have the following diagram of embeddings:
$$\xymatrix{R_F\lbr \La\rbr _F\ar@{}[r]|-*[@]{\subseteq} \ar@{}[d]|-*[@]{\subseteq}& R_F[\tfrac{1}{2\ttt x_\ga}]\lbr \La\rbr _F \ar@{}[d]|-*[@]{\subseteq} \\
R_F\lbr \La\rbr _F[\tfrac{1}{x_\Sigma}]\ar@{}[r]|-*[@]{\subseteq} &  R_F[\tfrac{1}{2\ttt x_\ga}]\lbr \La\rbr _F[\tfrac{1}{x_\Sigma}].}$$
Here ${x_\Sigma}:=\prod_{\al\in \Sigma}{x_\al}$.  By \cite[Corollary 3.4]{CZZ}, we see that this  square is cartesian, so 
$$q_w\in R_F[\tfrac{1}{2\ttt x_\ga}]\lbr \La\rbr _F\cap R_F\lbr \La\rbr _F[\tfrac{1}{x_\Sigma}]=R_F\lbr \La\rbr _F.$$

If $\kappa^F\neq 0$, replacing $2\ttt x_\ga$ by $a\ttt \Th$ in the above diagram, and inverting $\{\kappa_\al|\al\in \Sigma\}$ in all four rings, then by \cite[Corollary 3.4]{CZZ}, inside $R_F[\tfrac{1}{a\ttt\Th}]\lbr \La\rbr _F^\kappa[\tfrac{1}{x_{\Sigma}}]$, we have $$q_w\in R_F[\tfrac{1}{a\ttt \Th}]\lbr \La\rbr _F\cap R_F\lbr \La\rbr ^\kappa_F[\tfrac{1}{x_\al }|\al\in \Sigma]=R_F\lbr \La\rbr _F^\kappa.$$
Here we are using the fact that $R_F[\tfrac{1}{a\ttt\Th}]\lbr\La\rbr_F=R_F[\tfrac{1}{a\ttt\Th}]\lbr\La\rbr_F^\kappa$ by Lemma \ref{lem:invertelem}. 
\end{proof}

The following theorem was proved in \cite[Proposition 12.2]{Gin} if $F$ is a multiplicative FGL or the additive FGL (see also \cite[Theorem 7.2.16]{CG}), and was partially proved in \cite[Proposition 9.1]{HMSZ} if $a$ is invertible in $R$. 

 \begin{theorem} \label{thm:action}Under the hypothesis in Theorem \ref{thm:mainhecke}, then $\HF$ is isomorphic to the subalgebra of $\End_{R_F}(\FFAK)$ generated by $T_i$ for all $i\in [n]$ and by multiplications of elements in $\FFAK$. 
\end{theorem}

\begin{proof} It suffices to show that the action of $\HF$ on $\FFAK$ is faithful. Let $f=\sum_{w\in W}c_wT_{I_w}\in \HF$ with $c_w\in R_F\lbr \La\rbr _F^\kappa$ such that $f$ acts trivially on $\FFAK$. In particular $f\cdot \tau_{I_v}(u_0)=0$ for all $v\in W$. Via  the inclusion $\HF\subseteq \CHF$ in Corollary \ref{cor:HFCHF}, it induces $f':=\sum_{w\in W}c_w\mathcal{T}_{I_w}\in \CHF$, and by \eqref{eq:HFCHF}, we have $\sum_{w\in W}c_w\tau_{I_w}\tau_{I_v}(u_0)=f'\cdot \tau_{I_v}(u_0)=0$ by viewing $\tau_{I_v}(u_0)$ in $\CR\lbr\La\rbr_F$. 
By Lemma \ref{lem:u0tau},  the matrix $(\tau_{I_w}\tau_{I_v}(u_0))_{(w,v)\in W\times W}$ is invertible in $\CR\lbr \La\rbr _F$, so $c_w=0$ for all $w\in W$.
\end{proof}

With weaker restriction on the coefficient ring $R$, using functoriality, we could prove a weaker property of $\HF$ for any FGL $F$ 
\begin{corollary} \label{cor:kappaij'}Suppose  $\FFA$ is $\Sigma$-regular and let $F$ be any one-dimensional FGL,  then    $\HF$ is a left $\FFAK$-module with basis $\{T_{I_w}\}_{w\in W}$.
\end{corollary}
\begin{proof} 
Let $F_u$ be the universal FGL over the Lazard ring $\mathbb{L}$, and let $\La_r\subseteq \La$ be the root lattice. In this case the hypothesis of Theorem \ref{thm:mainhecke} is satisfied, hence its conclusion holds for $\mathbf{H}_{F_u}^{\La_r}$. In particular, the coefficients $\tau^{I_w,F_u}_{ij} $ appearing in Proposition \ref{prop:hecke} belong to $\mathbb{L}_{F_u}\lbr\La_r\rbr_{F_u}^{\ka}$. We then prove the the general case. Using the identities (1)--(3) in Proposition \ref{prop:hecke} and Corollary \ref{cor:commute2}, it suffices to prove that the coefficients $\tau^{I_w,F}_{ij}$ belong to $\FFAK$. 

 By \cite[Proposition 5.9]{CZZ}, We have the following commutative diagram
$$\xymatrix{\mathbb{L}_{F_u}\lbr\La_r\rbr_{F_u}^{\ka}\ar[r]^\phi\ar@{}[d]|-*[@]{\subseteq} & \FFAK\ar@{}[d]|-*[@]{\subseteq}\\
        (Q^{F_u})'\ar[r]^\phi & (Q^F)', }$$
        where $\phi$ is the map induced by the unique map of FGLs from $(\mathbb{L},F_u)$ to $(R,F)$. It also induces a map $\phi:\mathbf{H}_{F_u}\to \HF$ sending $T_i^{F_u}$ to $T_i^F$, hence it maps $\tau^{I_w, F_u}_{ij}$ to $\tau^{I_w, F}_{ij}$. So $\tau^{I_w}_{ij}\in \FFAK$ for any $i,j\in [n]$ and $w\in W$. 
\end{proof}

%\begin{lemma}If $2$ is regular in $R$, then $x_\al-x_{-\al}$ is regular in $\FA$ for any root $\al$.
%\end{lemma}
%\begin{proof}By \cite[Lemma 2.2]{CZZ}, we know that $\FA$ is $\Sigma$-regular, in particular, $x_\al$ is regular. Note that $x_\al-x_{-\al}$ is divisible by $x_\al$ and the leading term of $\tfrac{x_\al-x_{-\al}}{x_\al}$ is 2, so $x_\al-x_{-\al}$ is regular.
%\end{proof}
The following corollary generalizes \cite[Theorem 6.5]{Lus} and \cite[Theorem 7.1.14]{CG}.
\begin{corollary} 
\label{cor:center} Assume that $R$ has characteristic 0, then $(R_F\lbr\La\rbr_F^\kappa)^W$ is equal to the center of $\HF$.
\end{corollary}
\begin{proof}  Since $R$ has characteristic 0, $\FA$ is regular. By Corollary \ref{cor:kappaij'}, $\{T_{I_w}\}_{w\in W}$ is a basis of $\HF$. By Proposition \ref{prop:hecke}.(1), if $q\in (R_F\lbr\La\rbr_F^\kappa)^W$, then $s_i(q)=q$ and $\De_i(q)=0$ for any $i\in [n]$. Then $q$ belongs to the center of $\HF$.

Let $z=\sum_{w\in W}q_wT_{I_w}$ be in the center of $\HF$ with $q_w\in \FFAK$, we show that $q_w=0$ unless $w=e$ and $q_e\in (R\lbr\La\rbr_F^\kappa)^W$. Here $e\in W$ is the identity element. We proceed by decreasing induction on $\ell(w)$. Let $w_0$ be the longest element, and let $\al$ be any root, then $w_0(\al)\neq \al$. So $x_\al z=z x_\al=\sum_{w\in W}q_wT_{I_w}x_\al$. By Corollary \ref{cor:commute2}, comparing the coefficients of $T_{I_{w_0}}$ of the two sides, we see that $q_{w_0}x_\al=q_{w_0}x_{w_0(\al)}$, hence $q_{w_0}=0$.

Now suppose $q_w=0$ for $w$ such that $\ell(w)>l$. Let $v\in W$ with length $l$. We choose $\beta\in \Sigma^+$ such that $v(\beta)\neq  \beta$. We have 
$$\sum_{\ell(w)\le l}q_wx_\beta T_{I_w}=x_\beta z=zx_\beta=\sum_{\ell(w)\le l}q_wT_{I_w}x_\beta=\sum_{\ell(w)\le l}q_w\sum_{u\le w}\phi_{I_w,I_u}(x_\beta)T_{I_u}.$$
Comparing the coefficients of $T_{I_v}$, we see that $q_vx_\beta=q_vv(x_\beta)$, so $q_v=0$. So we have $q_w=0$ for $w\neq e$, i.e., $z\in \FFAK$. We then have 
$$zT_i=T_iz=s_i(z)T_i+\vartheta_i\De_i(z),$$
hence $(z-s_i(z))T_i=\vartheta_i\De_i(z)$. Since $\{T_{I_w}\}_{w\in W}$ is linearly independent,  $z=s_i(z)$, and $z\in (\FFAK)^W$. The proof is finished.

\end{proof}

%%%%%%%%%%%%%%%%%%%%%%%
\section{Normalization of formal group laws}
In the remaining part of this paper,  we assume that $a=a_{11}$ is invertible in $R$. In this section we study the notion of the so-called normal formal group law.  Recall that $\imath_Fx$ is the formal inverse of $x$, i.e., $x+_F\imath_Fx=0$. Define 
\begin{equation}\label{eq:h(x)}
h(x)=\tfrac{\imath_Fx+x}{\imath_Fx} ~\text{ so that }~ \imath_Fx=\tfrac{x}{-1+h(x)}.
\end{equation} The power series $h(x)$ has leading term $-ax$, so it has a composition inverse. That is, there is a power series $f(x)\in xR\lbr x\rbr $ such that 
$f(h(x))=h(f(x))=x.$

Recall that a homomorphism of formal group laws $g:F\to G$ is a power series $g(x)\in R\lbr x\rbr $ such that 
$$g(x)+_Gg(y)=g(x+_Fy).$$ 

\begin{lemma}\label{lem:fglkey}
If $a$ is invertible in $R$, then there exists a FGL $\tilF$ over $R$ such that $h:F\to \tilF$ is an isomorphism with   inverse $f:\tilF\to F$.  
Moreover, $\imath_{\tilF}x=\tfrac{x}{x-1}$, $\kappa^\tilF(x)=1$ and $\mu_\tilF (x)=\tfrac{1}{1-x}$.
\end{lemma}
\begin{proof}
Define $x+_{\tilF}y=h(f(x)+_Ff(y)),$ then it is straightforward to show that $\tilF$ is a FGL. Moreover,  $$h(x)+_\tilF h(y)=h [f(h(x))+_Ff(h(y))]=h(x+_Fy),$$
so $h:F\to \tilF$ defines a homomorphism of FGLs. It follows immediately that  $h$ is an isomorphism with inverse $ f:\tilF\to F$.

 Next we compute $\imath_\tilF x$. Since
$$0=h(f(x)+_F \imath_Ff(x))=h(f(x)+_Ff(h(\imath_Ff(x)))),$$
we have $$\imath_\tilF x=h(\imath_Ff(x))\overset{\eqref{eq:h(x)}}{=}h(\tfrac{f(x)}{h(f(x))-1})=h(\tfrac{f(x)}{x-1}).$$
 Let $z=f(x)$ hence $x=h(z)$. We have 
\begin{gather*}
1 =z\cdot\imath_F z\cdot \tfrac{1}{z\cdot\imath_F z}\overset{\eqref{eq:h(x)}}={z\cdot\imath_Fz}\cdot \tfrac{h(\imath_Fz)-1}{\imath_Fz}\tfrac{h(z)-1}{z}\\
=(h(z)-1)(h(\imath_Fz)-1)=(x-1)(h(\tfrac{z}{h(z)-1})-1)\\
=(x-1)(h(\tfrac{f(x)}{x-1})-1)=(x-1)(\imath_\tilF x-1).
\end{gather*}
So $\imath_\tilF x=\tfrac{x}{x-1}$,  $\kappa^F(x)=\tfrac{1}{x}+\tfrac{1}{\imath_\tilF x}=1$ and $\mu_\tilF(x)=\tfrac{\imath_\tilF(x)}{-x}=\tfrac{1}{1-x}$.
\end{proof}

\begin{definition}Suppose that $a$ is invertible in $R$. We say that a FGL $F$ is \textit{normal} if $\imath_F(x)=\tfrac{x}{x-1}$. If $F$ is not normal but $a$ is invertible in $R$, then the associated normal formal group law $\tilF$ as in Lemma \ref{lem:fglkey} exists, and we call it the \textit{normalization} of $F$. 
\end{definition}
Clearly $F$ is normalizable if and only if $a$ is invertible in $F$. Moreover, if $F$ is normal, then its normalization is itself since in this case $h(x)=x$.

\begin{example}For multiplicative formal group law $F_m(x, y)=x+y+a xy$ with $a$ invertible in $R$, we have $\imath_{F_m}x=\tfrac{-x}{1+ax}$, so  $h_{F_m}(x)=-ax$. So its composition inverse is  $f_{F_m}(u)=-\tfrac{1}{a}x$. Therefore, $$\tilF_m(x,y)=h_{F_m}(f_{F_m}(x)+_{F_m}f_{F_m}(y))=x+y-xy.$$ This justifies the name ``normalization''.
\end{example}

%%%%%%%%%%%%%%%%%%%%%%%%%%%%%%%
\section{formal affine Demazure algebras and formal affine Hecke algebras of normal formal group laws}
In the present section, we apply the normalization of FGL to simplify the notations of the formal affine Demazure algebra $\DF$ and the formal affine Hecke algebra $\HF$. We assume that  $a$ is invertible in $R$. Recall that $h(x)=\tfrac{\imath_F x+x}{\imath_Fx}$ and $f$ is its composition inverse. 

By functoriality in \cite[Lemma 2.6]{CPZ}, we have canonical isomorphism of $R$-algebras 
$$\phi_f:R\lbr \La\rbr _F\to R\lbr \La\rbr _\tilF,\quad x_\la\mapsto f(x_\la)$$
with inverse $$\phi_h:R\lbr \La\rbr _\tilF \to R\lbr \La\rbr _F, \quad x_\la\mapsto h(x_\la).$$ 
This isomorphism extends to $R_F\lbr\La\rbr_F\cong R_\tilF\lbr\La\rbr_\tilF$, $Q_W^F\cong Q_W^{\tilF}$ and $(Q_W^F)'\cong (Q_W^\tilF)'$ immediately. Using the identity $x_i\ka^F_i=h(x_{i})$, it is straightforward to check that 
\begin{eqnarray}
\label{eq:xi}\phi_f(X^F_i)&=&\tfrac{1}{f(x_i)}-\tfrac{1}{f(x_i)}\de_i=\tfrac{x_i}{f(x_i)}X^{\tilF}_i,\\
\label{eq:xikai} \phi_f(x_i\ka^F_i)&=&\phi_f(h(x_i))=h(f(x_i))=x_{i},\\
\label{eq:kaprime}\phi_f((\ka^F_{ij})')&=&\ka^\tilF_{ij},\\
\label{eq:mui}\phi_f(\mu_F(x_\gamma))&=&\phi_f(\tfrac{1}{1-h(x_{\ga})})=\tfrac{1}{1-h(f(x_\ga))}=\tfrac{1}{1-x_\ga} =\mu_\tilF(x_{\gamma}),\\
\label{eq:th} \phi_f(\Th_F)&=&\mu_{\tilF}(x_{\gamma})-\mu_{\tilF}(x_{-\ga})=\Th_\tilF,\\
\label{eq:t}\phi_f(T^F_i)&=&T_i^\tilF.
 \end{eqnarray} 
We then have the following corollary:

\begin{corollary}\label{cor:3} If $a$ is invertible in $R$, then the map $\phi_f$ induces ring isomorphisms $\phi_f:\DF\to \mathbf{D}_{\tilF} $ and $ \phi_f:\HF \to \mathbf{H}_{\tilF}$, where the latter maps $T_i^F$ to $T_i^\tilF$. 
\end{corollary}
\begin{proof}
 Since $h(x)$ has leading term $ax$, $f(x)$ has leading term $\tfrac{1}{a}x$, therefore, $\tfrac{x_i}{f(x_i)}$ is invertible in $R\lbr \La\rbr _\tilF$. Hence, by \eqref{eq:xi},    $\phi_f:\DF\to \mathbf{D}_{\tilF}$ is surjective, hence, it is an  isomorphism. The conclusion for $\HF$ follows from \eqref{eq:t}.
\end{proof}
It's interesting to note that $\phi_f$ preserves $T_i$ but changes $X_i$. We do not have an explanation about that.
 
\begin{lemma}\label{lemma:2}\begin{itemize}
\item[(a)] If $F$ is normal, then 
\begin{equation}\label{eq:kappaij}
\kappa^F_{ij}=\tfrac{1}{x_{i+j}x_i}+\tfrac{1}{x_{i+j}x_j}-\tfrac{1}{x_{i+j}}-\tfrac{1}{x_ix_j}.
\end{equation}
Moreover, $\ka^F_{ij}=\ka^F_{ji}=\ka^F_{-i,-j}=\kappa^F_{-i,i+j}$.
\item[(b)] If $a$ is invertible in $R$, then   $(\kappa^F)'_{ij}\in R\lbr \La\rbr _F$. If in addition $F$ is normal, then $(\kappa^F_{ij})'=\kappa^F_{ij}$.
\end{itemize}
\end{lemma}
\begin{proof}
(a) By Lemma \ref{lem:fglkey}, we see that $\tfrac{1}{x_{-i}}=1-\tfrac{1}{x_i}.$ Substituting it into $\kappa^F_{ij}$, we get the formula (\ref{eq:kappaij}), which  is symmetric with respect to $i$ and $j$, so $\kappa^F_{ij}=\ka^F_{ji}$. Moreover,  we can verify that $\ka^F_{ij}=\ka^F_{-i,-j}=\ka^F_{-i,i+j}$.

(b) Identity \eqref{eq:kaprime} implies that  $(\kappa^F_{ij})'=\phi_h(\kappa^\tilF_{i,j})\in R\lbr \La\rbr _F$. In particular, if $F$ is normal, then $\kappa^F_i=1$ for all $i$, so by definition, $(\kappa^F_{ij})'=\kappa_{ij}^F.$ 
\end{proof}

 \medskip

We simplify the presentations of $\DF$ and $\HF$:
\begin{theorem}\label{thm:demazure}Suppose that $\FA$ is $\Sigma$-regular.  If $F$ is normal, the algebra $\DF$ satisfies the following properties:
\begin{itemize}
\item[(1)] For any $\al\in \Phi$, $(X^F_\al)^2=X^F_\al$.
\item[(2)]  If $(\al_i^\vee,\al_j)=(\al_j^\vee, \al_i)=-1$, so that  $m_{ij}=3$, then $$X^F_{jij}-X^F	_{iji}=\kappa^F_{ij}(X^F_{i}-X^F_j).$$ Moreover, $\kappa^F_{ij}$ is invariant under $s_i$ and $s_j$.
\item[(3)] If $(\al_i^\vee, \al_j)=-1$ and $(\al_j^\vee, \al_i)=-2$, so that $m_{ij}=4$, then $$X^F_{jiji}-X^F_{ijij}=(\kappa^F_{ij}+\kappa^F_{i,i+j})(X^F_{ij}-X^F_{ji}).$$
 Moreover, $\kappa^F_{ij}+\ka^F_{i,i+j}$ is invariant under $s_i$ and $s_j$.
\end{itemize}

\end{theorem}

\begin{proof} 
(1) It follows from Proposition \ref{prop:demazure} and the identity $\kappa_\al=1$. 

(2) By Lemma \ref{lemma:2}, we see that $\kappa^F_{ij}=\ka^F_{ji}$. Moreover, direct computation shows that $s_{i}(\ka^F_{ij})=s_{j}(\ka_{ij})=\ka^F_{ij}$. The conclusion then follows from Proposition \ref{prop:demazure}.

(3) By Lemma \ref{lemma:2}, we see that $\kappa^F_{i+2j,-j}=\ka^F_{-j,i+2j}=\kappa^F_{j,i+j}$. Moreover, direct computation shows that $\kappa^F_{ij}+\kappa^F_{i,i+j}$ is invariant under $s_i$ and $s_j$. Therefore, $\De_i(\kappa^F_{ij}+\kappa^F_{i,i+j})=\De_j(\kappa^F_{ij}+\kappa^F_{i,i+j})=0$. The conclusion then follows from Proposition \ref{prop:demazure}.
\end{proof}

%%%%%%%%%%%%%%%%%

\begin{theorem}\label{thm:hecke}Suppose that $\FFA$ is $\Sigma$-regular and that the root datum does not contain an irreducible  component of type $G_2$.  If $F$ is normal, then the algebra $\HF$ satisfies the following properties:
\begin{itemize}
\item[(1)] If $(\al_i^\vee,\al_j)=(\al_j^\vee,\al_i)=-1$, so that $m_{ij}=3$, then 
$$T^F_{jij}-T^F_{iji}=\Th_F^2\kappa^F_{ij}(T^F_i-T^F_j).$$
\item[(2)] If $(\al_i^\vee, \al_j)=-1$ and $(\al_j^\vee, \al_i)=-2$, so that $m_{ij}=4$, then 
$$T^F_{jiji}-T^F_{ijij}=\Th_F^2(\kappa^F_{ij}+\kappa^F_{j,i+j})(T^F_{ij}-T^F_{ji}). $$
 \item[(3)] $\FFA=R_F\lbr \La\rbr _F^\kappa=R_F\lbr \La\rbr _F^\sim$.
\item[(4)]  The set $\{T^F_{I_w}\}_{w\in W}$ is a basis of $\HF$ as a left $\FFA$-module.
\item[(5)] As an $R$-algebra, $\HF$ is generated by $\FFA$ and $T_i, i\in [n]$ with  relations (1), (2) and (4) of Proposition \ref{prop:hecke} and (1) and (2) of this theorem.
\end{itemize}

If $F$ is not normal but $a$ is invertible in $R$, then  $\HF$ satisfies
the above properties after replacing $\kappa^F_{ij}$ (resp. $\kappa^F_{j,i+j}$) by $(\kappa^F_{ij})'$ (resp. $(\kappa^F_{j,i+j})'$).
\end{theorem}
\begin{proof}

 (1): It follows from Proposition \ref{prop:hecke} and the identity $\kappa_\al=1$. 
 
 (2): It follows from  direct computation of $T_{jiji}-T_{ijij}$.
 
 (3): Since $\kappa_\al=1$ for all $\al$,  $\FFA=\FFAK$. Since we only consider root systems of type different from type $G_2$,  $m_{ij}\le 4$. Therefore, the elements $\tau^{I_w}_{ij}$ appeared in Proposition \ref{prop:hecke} belong to the set $\{\kappa_{ij}^F, \kappa^F_{ij}+\ka^F_{j,i+j}\}\subsetneq \FFA$. So $R_F\lbr \La\rbr _F^\sim=\FFA.$ 
 
 (4) and (5): These are direct consequence of (1)-(3) of this theorem  and Proposition \ref{prop:hecke}.
 
 If $F$ is not normal but $a$ is invertible, then its normalization $\tilF$ exists. By \eqref{eq:kaprime}, the isomorphism $\phi_f:\FFA\to R_\tilF\lbr \La\rbr_\tilF$ maps $(\kappa_{ij}^F)'$ to $\kappa_{ij}^\tilF$. Therefore, the $R$-algebra $\HF$ satisfies the properties (1)-(5) above after replacing $\kappa_{ij}^F$ by $(\kappa_{ij}^F)'$.
\end{proof}

\end{document}